\pdfoutput=1 
\documentclass[12pt,twoside,leqno]{article}
\usepackage{amsmath,amssymb,mathtools}
\usepackage{verbatim}
\usepackage{dsfont,bm,mathrsfs}
\usepackage{color}
\usepackage{graphicx}
\usepackage{amsthm}
\usepackage[shortlabels]{enumitem}
\usepackage{esint} 
\usepackage[title]{appendix}
\usepackage[T1]{fontenc}

\usepackage[dvipsnames]{xcolor}
\definecolor{darkgreen}{rgb}{0,0.75,0}
\definecolor{darkred}{rgb}{0.75,0,0}
\definecolor{darkmagenta}{rgb}{0.5,0,0.5}
\usepackage[colorlinks,citecolor=Emerald,linkcolor=NavyBlue,urlcolor=RoyalBlue]{hyperref}
\newtheorem{thm}{Theorem}[section]

\newtheorem{lem}[thm]{Lemma}
\newtheorem{prop}[thm]{Proposition}

\theoremstyle{definition}
\newtheorem{defn}[thm]{Definition}

\newtheorem{rmk}[thm]{Remark}

\numberwithin{equation}{section}
\numberwithin{figure}{section}

\allowdisplaybreaks

\newcommand{\norm}[1]{\lVert#1\rVert}

\newcommand{\abs}[1]{\lvert#1\rvert}
\newcommand{\absbig}[1]{\bigl\lvert#1\bigr\rvert}
\newcommand{\absBig}[1]{\Bigl\lvert#1\Bigr\rvert}
\newcommand{\absbigg}[1]{\biggl\lvert#1\biggr\rvert}
\newcommand{\absBigg}[1]{\Biggl\lvert#1\Biggr\rvert}
\newcommand{\trinorm}[1]{{\left\vert\kern-0.25ex\left\vert\kern-0.25ex\left\vert #1
    \right\vert\kern-0.25ex\right\vert\kern-0.25ex\right\vert}}
\newcommand{\indicator}[1]{\mathds{1}_{#1}} 

\makeatletter
\DeclareRobustCommand\widecheck[1]{{\mathpalette\@widecheck{#1}}}
\def\@widecheck#1#2{%
   \setbox\z@\hbox{\m@th$#1#2$}%
   \setbox\tw@\hbox{\m@th$#1%
      \widehat{%
         \vrule\@width\z@\@height\ht\z@
         \vrule\@height\z@\@width\wd\z@}$}%
   \dp\tw@-\ht\z@
   \@tempdima\ht\z@ \advance\@tempdima2\ht\tw@ \divide\@tempdima\thr@@
   \setbox\tw@\hbox{%
      \raise\@tempdima\hbox{\scalebox{1}[-1]{\lower\@tempdima\box\tw@}}}%
   {\ooalign{\box\tw@ \cr \box\z@}}}
\makeatother

\newcommand{\measure}{\mathfrak{m}}
\newcommand{\metric}{d}

\newcommand{\mr}[1]{{\tt \href{http://mathscinet.ams.org/mathscinet-getitem?mr=#1}{MR#1}}}
\newcommand{\arxiv}[1]{{\tt \href{http://arxiv.org/abs/#1}{arXiv:#1}}}

\begin{document}

	\font\titlefont=cmbx14 scaled\magstep1
	\title{\titlefont Higher-order differentiability of Korevaar--Schoen energy forms and energy measures}

	\author{ 
	Ryosuke Shimizu\footnote{\scriptsize Graduate School of Informatics, Kyoto University, Kyoto, Japan. \texttt{r.shimizu@acs.i.kyoto-u.ac.jp}}
	}
	\date{October 6, 2026}
	\maketitle
	\vspace{-0.5cm} 
	\begin{abstract}
		In this paper, we investigate the differentiability of Korevaar--Schoen $p$-energy forms and the associated $p$-energy measures. 
		We obtain higher-order derivatives by virtue of explicit realizations of Korevaar--Schoen $p$-energy forms and the associated $p$-energy measures as subsequential pointwise limits of certain double integrals.
	\end{abstract}

\section{Introduction}\label{sec:intro} 
Korevaar and Schoen \cite{KS93} have introduced Dirichlet $p$-energy functionals and the associated Sobolev spaces for maps $u$ from a smooth Riemannian manifold $X$ to a complete metric space $(Y,d_{Y})$, roughly speaking, by interpreting the limit as $\varepsilon \to 0$ of  the ball averages 
\[
\frac{1}{\measure(B(x,\varepsilon))}\int_{B(x,\varepsilon)}\frac{d_{Y}(u(x),u(x+y))}{\varepsilon^{p}}\,\measure(dy)
\] 
as the the $p$-th power of the energy density of $u$ at $x$, which corresponds to $\abs{\nabla u(x)}^{p}$ when $Y = \mathbb{R}$. 
Here, $1 < p < \infty$, $X$ is equipped with the usual Riemannian distance $d$ and the volume measure $\measure$, $B(x,\varepsilon)$ is the open ball in $X$ centered at $x \in X$ with radius $\varepsilon > 0$. 
Although one of the main motivations of their work \cite{KS93} is to investigate harmonic maps into non-smooth metric spaces, even in the case of $Y = \mathbb{R}$, their approach has turned out to be useful in characterizing the $(1,p)$-Sobolev spaces $\mathcal{F}_{p}$ on metric measure spaces including fractals; see, e.g., \cite{ARB25,AEBS25+,Bau24,CGYZ26,KS26,KM98,MS25}. 
For example, if $X$ is the Sierpi\'nski gasket or the Sierpi\'nski carpet (with certain choices of the metric $d$ and the measure $\measure$) instead of a smooth Riemannian manifold, then there exist an intrinsic parameter $d_{\mathrm{w},p}$, called the \emph{$p$-walk dimension}, such that the $(1,p)$-Sobolev space $\mathcal{F}_{p}$ constructed via graph approximations and discrete energy forms on them in \cite{CGQ22,HPS04,Kig23,MS25} can be characterized as the set of all functions $u \in L^p(X,\measure)$ satisfying 
\[
\limsup_{\varepsilon \to 0}\int_{X}\frac{1}{\measure(B(y,\varepsilon))}\int_{B(y,\varepsilon)}\frac{\abs{u(x) - u(y)}^{p}}{\varepsilon^{d_{\mathrm{w},p}}}\,\measure(dx)\,\measure(dy) < \infty. 
\]
Furthermore, there exists a sequence $\{ \varepsilon_n \}_{n}$ in $(0,\infty)$ converging to $0$ such that the limit 
\[
\lim_{n \to \infty}\int_{X}\frac{1}{\measure(B(y,\varepsilon_{n}))}\int_{B(y,\varepsilon_{n})}\frac{\abs{u(x) - u(y)}^{p}}{\varepsilon_{n}^{d_{\mathrm{w},p}}}\,\measure(dx)\,\measure(dy)
\eqqcolon \mathcal{E}_{p}(u)
\]
exists for all Sobolev functions $u \in \mathcal{F}_{p}$. 
In addition, the associated $p$-energy measure $\Gamma_{p}\langle u \rangle$ can be introduced as an analogue of $\abs{\nabla u}^{p}\,dx$ in the Euclidean case: There exists a unique positive Radon measure $\Gamma_{p}\langle u \rangle$ on $X$ such that for all compactly supported continuous functions $\varphi$ on $X$, 
\begin{align*}
&\int_{X}\varphi\,d\Gamma_{p}\langle u \rangle \\
&= \lim_{n \to \infty}\int_{X}\frac{1}{\measure(B(y,\varepsilon_{n}))}\int_{B(y,\varepsilon_{n})}\frac{\abs{u(x) - u(y)}^{p}}{\varepsilon_{n}^{d_{\mathrm{w},p}}}\varphi(x)\,\measure(dx)\,\measure(dy). 
\end{align*}
See \cite{KS26} for further details on these characterizations. 

The main focus of this paper is the differentiability of the Korevaar--Schoen energy forms $(\mathcal{E}_{p},\mathcal{F}_{p})$ and the associated energy measures $\Gamma_{p}$. 
The constructions of $p$-energy forms on fractals in the previous studies \cite{CGQ22,HPS04,Kig23,MS25} include many steps such as the operation of taking a subsequential scaling limit of discrete $p$-energy forms.
Due to these operations, the differentiability of the resulting energy forms are not evident. 
Indeed, in order to develop feasible non-linear potential theory on fractals, Strichartz and Wong \cite{SW04} have proposed an approach based on the subderivatives of the $p$-energy forms on the Sierpi\'nski gaskets constructed in the earlier work \cite{HPS04}. 
It was subsequently observed in \cite{KS.gc} that the $L^p$-analogue of parallelogram law called the $p$-Clarkson's inequality implies the first-order Fr\'echet differentiability of energy forms. (It was also shown that the previous constructions in \cite{CGQ22,HPS04,Kig23,MS25} can be modified to satisfy $p$-Clarkson's inequality; see \cite[Sections 3,4, and 8]{KS.gc}.)
This led to the development of the foundations of non-linear potential theory based on their first-order derivatives; see \cite[Section 6]{KS.gc} for details.  

Although first-order differentiability is now well understood, it remains unclear whether there exist natural functional inequalities for the energy forms, analogous to $p$-Clarkson's inequality for the first-order differentiability, that guarantee higher-order Fr\'echet differentiability. 
The aim of this paper is to establish higher-order differentiability of the Korevaar--Schoen energy forms and of the associated energy measures, as in the case of the $L^p$-norm proved in \cite{Sun67}. 
To this end, we take advantage of the construction of the Korevaar--Schoen energy forms/measures established in \cite{KS26}.  

This paper is organized as follows. 
In Section \ref{sec:preli}, we recall the construction of Korevaar--Schoen energy forms $(\mathcal{E}_{p},\mathcal{F}_{p})$ and the associated energy measures $\Gamma_{p}$ established in \cite{KS26}, and construct multilinear energy forms $\mathcal{E}_{p}^{(l)}$ that will describe the higher-order derivatives of energy forms. 
We prove our main result for $(\mathcal{E}_{p},\mathcal{F}_{p})$ in Section \ref{sec:higher.form}. 
In Section \ref{sec:higher.measures}, we construct multilinear energy measures $\Gamma_{p}^{(l)}$ and show the main result for $\Gamma_{p}$. 

\noindent 
\textbf{Disclosure of AI use.} 
Generative AI tools (ChatGPT and Microsoft Copilot) were used solely to check grammar and references, and to assist with language editing and refinement.

\noindent 
\textbf{Acknowledgments.} This work originated from a question posed by Naotaka Kajino. The author thanks him for the question and for several illuminating discussions. 
The author is supported by JSPS KAKENHI Grant Number JP26K16997.

\section{Korevaar--Schoen energy forms and energy measures}\label{sec:preli} 
Let us recall some results on the Korevaar--Schoen $p$-energy forms from \cite{KS26}. 
Throughout this paper, $1 < p < \infty$, $(X,\metric)$ is a separable metric space with $\#X \ge 2$, $\measure$ is a $\sigma$-finite Borel measure on $X$ with full topological support, and $\bm{k} = \{ k_{r} \}_{r > 0}$ is a family of $[0,\infty]$-valued Borel measurable functions on $X \times X$. 
See \cite[Example 6.3.14]{KS26} for several concrete examples of $\bm{k}$.  
For $u \in L^p(X,\measure)$ and $r > 0$, define 
\[
E_{p,r}(u) \coloneqq \int_{X}\int_{X}\abs{u(x) - u(y)}^{p}k_{r}(x,y)\,\measure(dx)\,\measure(dy). 
\]
Let $\mathcal{F}_{p}$ denote the associated function space given by 
\begin{equation}\label{eq:KSspace}
	\mathcal{F}_{p} \coloneqq \biggl\{ u \in L^{p}(X,\measure) \biggm| (\abs{u}_{\mathcal{F}_{p}})^{p} \coloneqq \sup_{r > 0}E_{p,r}(u) < \infty \biggr\},  
\end{equation}
which coincides with $B_{p,\infty}^{\bm{k}}$ in \cite[Definition 6.3.1]{KS26}. 

Next let us introduce an important condition called the \emph{weak monotonicity property} (see \cite[(WM)$_{p,\bm{k}}$ in Definition 6.3.5]{KS26}) for $\bm{k}$. 
We say that $\bm{k}$ satisfies the weak monotonicity property if and only if there is $C > 0$ such that for any $u \in \mathcal{F}_{p}$, 
\begin{equation}\label{eq:WM}
	 \sup_{r > 0}E_{p,r}(u) \le C\liminf_{r \to 0+}E_{p,r}(u). 
\end{equation} 
If $\bm{k}$ satisfies the weak monotonicity property \eqref{eq:WM}, then the function space $\mathcal{F}_{p}$ equipped with the norm $\norm{\,\cdot\,}_{L^p} + \abs{\,\cdot\,}_{\mathcal{F}_{p}}$ is a reflexive and separable Banach space \cite[Theorem~6.3.6]{KS26}. 

Let us return to setting up the notation. 
Let $\lceil p \rceil \in \mathbb{N}$ be the unique positive integer satisfying $\lceil p \rceil - 1 < p \le \lceil p \rceil$ and set $m(p) \coloneqq \lceil p \rceil - 1 \in \mathbb{N}$.
For $l \in \mathbb{N}$ with $1 \le l \le m(p)$ and $u,v_{1},v_{2},\dots,v_{l} \in \mathcal{F}_{p}$, define 
\begin{align*}
	&E_{p,r}^{(l)}(u; v_{1},v_{2},\dots,v_{l}) \\
	&\coloneqq \int_{X}\int_{X}\Biggl[\abs{u(x) - u(y)}^{p-l}\bigl(\mathrm{sgn}(u(x)-u(y))\bigr)^{l} \\
	&\qquad\qquad  \times \Biggl(\prod_{j=1}^{l}(v_{j}(x) - v_{j}(y))\Biggr)k_{r}(x,y)\Biggr]\,\measure(dx)\,\measure(dy). 
\end{align*} 
Here we define $\mathrm{sgn}(a) \coloneqq \indicator{(0,\infty)}(a) - \indicator{(0,\infty)}(-a)$ for $a \in \mathbb{R}$, where $\indicator{A}$ denotes the indicator function of the set $A$.   
If $p \in \mathbb{N}$, then we also define 
\begin{align*}
&E_{p,r}^{(p)}(v_{1},v_{2},\dots,v_{p}) \\
&\coloneqq \int_{X}\int_{X}\left(\prod_{j=1}^{p}(v_{j}(x) - v_{j}(y))\right)k_{r}(x,y)\, \measure(dx)\,\measure(dy). 
\end{align*}
For ease of notation, we write $E_{p,r}^{(0)} \coloneqq E_{p,r}$, and, when $p \in \mathbb{N}$, we also write $E_{p,r}^{(l)}(u;\,\cdot\,) \coloneqq E_{p,r}^{(p)}(\,\cdot\,)$.  
Also, for $a,b \in \mathbb{R}$, we define $a^{+} \coloneqq \max\{ a,0 \}$, $a \vee b \coloneqq \max\{ a,b \}$ and $a \wedge b \coloneqq \min\{ a,b \}$. 

\begin{lem}\label{lem:-preform.basic}
	Let $r > 0$ and $l \in \mathbb{Z}$ with $0 \le l \le p$. 
	\begin{enumerate}[\rm(a)]
		\item\label{it:-preform.homog} For any $a \in \mathbb{R}$ and any $u \in \mathcal{F}_{p}$, 
		\[
		E_{p,r}^{(l)}(au;\,\cdot\,) = \abs{a}^{p-l}\mathrm{sgn}(a)^{l}E_{p,r}^{(l)}(u;\,\cdot\,). 
		\] 
		\item\label{it:-preform.symform} If $l \ge 1$ and $u \in \mathcal{F}_{p}$, then $E_{p,r}^{(l)}(u;\,\cdot\,) \colon (\mathcal{F}_{p})^{l} \to \mathbb{R}$ is an $l$-multilinear symmetric form. 
		\item For any $u,v_{1},v_{2},\dots,v_{l} \in \mathcal{F}_{p}$, 
			\begin{align}\label{eq:pre-holder}
				&\absBig{E_{p,r}^{(l)}(u; v_{1},v_{2},\dots,v_{l})} \\
				&\le E_{p,r}(u)^{\frac{p-l}{p}}\prod_{j=1}^{l}E_{p,r}(v_{j})^{\frac{1}{p}}
				\le \bigl(\abs{u}_{\mathcal{F}_{p}}\bigr)^{p-l}\prod_{j=1}^{l}\abs{v_{j}}_{\mathcal{F}_{p}}. \notag 
			\end{align}
		\item There exists $C_{p} \in (0,\infty)$ determined solely and explicitly in $p$ such that for any $u,u',v_{1},v_{2},\dots,v_{l} \in \mathcal{F}_{p}$, 
			\begin{align}\label{eq:pre-conti.nonlinear}
				&\absBig{E_{p,r}^{(l)}(u;v_{1},v_{2},\dots,v_{l}) - E_{p,r}^{(l)}(u';v_{1},v_{2},\dots,v_{l})}  \\
				&\le C_{p}\Bigl[E_{p,r}(u)^{\frac{1}{p}} \vee E_{p,r}(u')^{\frac{1}{p}}\Bigr]^{(p-l-1)^{+}}\bigl(E_{p,r}(u-u')^{\frac{1}{p}}\bigr)^{(p-l) \wedge 1} \notag \\
				&\quad \times \Biggl(\prod_{j=1}^{l}E_{p,r}(v_j)^{\frac{1}{p}}\Biggr) \notag \\
				&\le C_{p}\Bigl[\abs{u}_{\mathcal{F}_{p}} \vee \abs{u'}_{\mathcal{F}_{p}}\Bigr]^{(p-l-1)^{+}}\bigl(\abs{u - u'}_{\mathcal{F}_{p}}\bigr)^{(p-l) \wedge 1}\prod_{j=1}^{l}\abs{v_{j}}_{\mathcal{F}_{p}}. \notag 
			\end{align}
	\end{enumerate}
\end{lem}
\begin{proof}
	The assertions \ref{it:-preform.homog} and \ref{it:-preform.symform} are obvious from the definition and \eqref{eq:pre-holder} immediately follows from H\"{o}lder's inequality. 
	The estimate \eqref{eq:pre-conti.nonlinear} follows from the following continuity estimate on $a \mapsto \abs{a}^{q}\mathrm{sgn}(a)$ for $q > 1$ (see \cite[Lemma~6.3.7]{KS26})
	\begin{equation}\label{eq:powerestimate}
		\absbig{\abs{a}^{q} - \abs{b}^{q}} 
		\le \absbig{a^{q} - b^{q}} \le 2\bigl(1 \vee (q-1)\bigr)(a \vee b)^{(q-2)^{+}}\abs{a-b}^{(q-1) \wedge 1}, 
	\end{equation}
	so the proof is completed. 
\end{proof}

\begin{lem}\label{lem:subsec}
	Assume that the weak monotonicity property \eqref{eq:WM} holds. 
	Let $\{ \widetilde{r}_{n} \}_{n \ge 1}$ be a sequence in $(0,\infty)$ such that $\lim_{n \to \infty}\widetilde{r}_{n} = 0$. 
	Then there exists a subsequence $\{ r_{n} \}_{n}$ of $\{ \widetilde{r}_{n} \}_{n}$ such that, for any $l \in \mathbb{Z}$ with $0 \le l \le p$ and any $u,v_{1},v_{2},\dots,v_{l} \in \mathcal{F}_{p}$, $\bigl\{ E_{p,r_{n}}^{(l)}(u; v_{1},v_{2},\dots,v_{l}) \bigr\}_{n}$ is convergent in $\mathbb{R}$. 
\end{lem}
\begin{proof}
	Fix a countable dense subset $\mathcal{C}$ of $\mathcal{F}_{p}$.  
	In view of \eqref{eq:pre-holder}, a standard diagonal procedure yields a subsequence $\{ r_{n} \}_{n}$ such that the sequence $\bigl\{ E_{p,r_n}^{(l)}(u; v_{1},v_{2},\dots,v_{l}) \bigr\}_{n}$ is convergent for all $0 \le l \le p$ and $u,v_{1},v_{2},\dots,v_{l} \in \mathcal{C}$.
	
	Let $0 < \varepsilon < 1$ and $u,v_{1},v_{2},\dots,v_{l} \in \mathcal{F}_{p}$. 
	Choose $\hat{u},\hat{v}_{1},\hat{v}_{2},\dots,\hat{v}_{l} \in \mathcal{C}$ so that 
	\[
	\max\bigl\{ \abs{u - \hat{u}}_{\mathcal{F}_{p}}, \abs{v_{1} - \hat{v}_{1}}_{\mathcal{F}_{p}}, \abs{v_{2} - \hat{v}_{2}}_{\mathcal{F}_{p}}, \dots, \abs{v_{l} - \hat{v}_{l}}_{\mathcal{F}_{p}} \bigr\} 
	< \varepsilon. 
	\]
	Then one can see that $\bigl\{ E_{p,r_n}^{(l)}(u; v_{1},v_{2},\dots,v_{l})  \bigr\}_{n}$ is convergent in $\mathbb{R}$ as follows: 
	For any $r > 0$, by \eqref{eq:pre-conti.nonlinear} and \eqref{eq:pre-holder}, 
	\begin{align*}
		&\absBig{E_{p,r}^{(l)}(u; v_{1},v_{2},\dots,v_{l}) - E_{p,r}^{(l)}(\hat{u}; \hat{v}_{1},\hat{v}_{2},\dots,\hat{v}_{l})} \\
		&\quad \le \absBig{E_{p,r}^{(l)}(u; v_{1},v_{2},\dots,v_{l}) - E_{p,r}^{(l)}(\hat{u}; v_{1},v_{2},\dots,v_{l})} \\
		&\quad\quad + \absBig{E_{p,r}^{(l)}(\hat{u}; v_{1},v_{2},\dots,v_{l}) - E_{p,r}^{(l)}(\hat{u}; \hat{v}_{1},v_{2},\dots,v_{l})} \\
		&\quad\quad + \absBig{E_{p,r}^{(l)}(\hat{u}; \hat{v}_{1},v_{2},\dots,v_{l}) - E_{p,r}^{(l)}(\hat{u}; \hat{v}_{1},\hat{v}_{2},\dots,v_{l})} \\
		&\quad\quad + \cdots + \absBig{E_{p,r}^{(l)}(\hat{u}; \hat{v}_{1},\hat{v}_{2},\dots,\hat{v}_{l-1},v_{l}) - E_{p,r}^{(l)}(\hat{u}; \hat{v}_{1},\hat{v}_{2},\dots,\hat{v}_{l})}  \\
		&\quad \le C\varepsilon^{(p-l) \wedge 1}, 
	\end{align*} 
	where $C \in (0,\infty)$ is a constant independent of $r$. 
	Using this estimate, we see that for any $n,n'$, 
	\begin{align*}
		&\absBig{E_{p,r_n}^{(l)}(u; v_{1},v_{2},\dots,v_{l}) - E_{p,r_{n'}}^{(l)}(u; v_{1},v_{2},\dots,v_{l})} \\
		&\quad \le 2C\varepsilon^{(p-l) \wedge 1} + \absBig{E_{p,r_n}^{(l)}(\hat{u}; \hat{v}_{1},\hat{v}_{2},\dots,\hat{v}_{l}) - E_{p,r_{n'}}^{(l)}(\hat{u}; \hat{v}_{1},\hat{v}_{2},\dots,\hat{v}_{l})}. 
	\end{align*}
	Hence 
	\[
	\limsup_{n \wedge n' \to \infty}\absBig{E_{p,r_n}^{(l)}(u; v_{1},v_{2},\dots,v_{l}) - E_{p,r_{n'}}^{(l)}(u; v_{1},v_{2},\dots,v_{l})} 
	\le 2C\varepsilon^{(p-l) \wedge 1}, 
	\]
	which means that $\bigl\{ E_{p,r_{n}}^{(l)}(u; v_{1},v_{2},\dots,v_{l}) \bigr\}_{n}$ is a Cauchy sequence in $\mathbb{R}$. 
\end{proof}

\begin{defn}\label{defn:KS}
	Assume that the weak monotonicity property \eqref{eq:WM} holds and that $\{ \widetilde{r}_{n} \}_{n \ge 1}$ is a sequence in $(0,\infty)$ converging to $0$. 
	Let $\{ r_{n} \}_{n}$ be a subsequence of $\{ \widetilde{r}_{n} \}_{n \ge 1}$ as given in Lemma \ref{lem:subsec}. 
	For $l \in \mathbb{Z}$ with $0 \le l \le p$, define $\mathcal{E}_{p}^{(l)} \colon (\mathcal{F}_{p})^{l+1} \to \mathbb{R}$ by 
	\begin{equation}\label{eq:defn-multiform}
		\mathcal{E}_{p}^{(l)}(u; v_{1},v_{2},\dots,v_{l}) \coloneqq \lim_{n\to\infty}E_{p,r_{n}}^{(l)}(u; v_{1},v_{2},\dots,v_{l}) 
	\end{equation}
	for $u,v_{1},v_{2},\dots,v_{l} \in \mathcal{F}_{p}$. 
	We write $\mathcal{E}_{p} \coloneqq \mathcal{E}_{p}^{(0)}$ and call $(\mathcal{E}_{p},\mathcal{F}_{p})$ the \emph{$\bm{k}$-Korevaar--Schoen $p$-energy form (on $(X,\measure)$) along $\{ r_{n} \}_{n}$} \cite[Definition 6.3.9]{KS26}. 
	From the weak monotonicity property \eqref{eq:WM}, the norm $\bigl(\norm{\,\cdot\,}_{L^p}^{p} + \mathcal{E}_{p}\bigr)^{1/p}$ is comparable with that of $\mathcal{F}_{p}$. 
\end{defn}
\begin{rmk}\label{rmk:choice}
	The case of $l = 0$ corresponds to the construction of Korevaar--Schoen $p$-energy form denoted by $(\mathcal{E}_{p}^{\bm{k}}, B_{p,\infty}^{\bm{k}})$ in \cite[Eq. (6.3.8) in Theorem 6.3.8]{KS26}. 
	Note that this construction depends on the choice of the sequence $\{ \widetilde{r}_{n} \}_{n}$ and its subsequence $\{ r_{n} \}_{n}$. 
\end{rmk}

The following proposition is immediate from Lemma \ref{lem:-preform.basic} and \eqref{eq:defn-multiform}. 
\begin{prop}\label{prop:-form.basic}
	Assume that the weak monotonicity property \eqref{eq:WM} holds and that $\{ \widetilde{r}_{n} \}_{n \ge 1}$ is a sequence in $(0,\infty)$ converging to $0$. 
	Let $\{ r_{n} \}_{n}$ be a subsequence of $\{ \widetilde{r}_{n} \}_{n \ge 1}$ as given in Lemma \ref{lem:subsec}, let $l \in \mathbb{Z}$ with $0 \le l \le p$ and let $\mathcal{E}_{p}^{(l)}$ be as defined in \eqref{eq:defn-multiform}. 
	\begin{enumerate}[\rm(a)]
		\item For any $a \in \mathbb{R}$ and any $u \in \mathcal{F}_{p}$, 
		\[
		\mathcal{E}_{p}^{(l)}(au;\,\cdot\,) = \abs{a}^{p-l}\mathrm{sgn}(a)^{l}\mathcal{E}_{p}^{(l)}(u;\,\cdot\,). 
		\] 
		\item If $l \ge 1$, then, for any $u \in \mathcal{F}_{p}$, $\mathcal{E}_{p}^{(l)}(u;\,\cdot\,) \colon (\mathcal{F}_{p})^{l} \to \mathbb{R}$ is an $l$-multilinear symmetric form. 
		\item For any $u,v_{1},v_{2},\dots,v_{l} \in \mathcal{F}_{p}$, 
			\begin{equation}\label{eq:holder}
				\absBig{\mathcal{E}_{p}^{(l)}(u; v_{1},v_{2},\dots,v_{l})}  
				\le \bigl(\abs{u}_{\mathcal{F}_{p}}\bigr)^{p-l}\prod_{j=1}^{l}\abs{v_{j}}_{\mathcal{F}_{p}}. 
			\end{equation} 
		\item For any $u,u',v_{1},v_{2},\dots,v_{l} \in \mathcal{F}_{p}$, 
			\begin{align}\label{eq:conti.nonlinear}
				&\absBig{\mathcal{E}_{p}^{(l)}(u;v_{1},v_{2},\dots,v_{l}) - \mathcal{E}_{p}^{(l)}(u';v_{1},v_{2},\dots,v_{l})} \\
				&\quad \le C_{p}\Bigl[\abs{u}_{\mathcal{F}_{p}} \vee \abs{u'}_{\mathcal{F}_{p}}\Bigr]^{(p-l-1)^{+}}\bigl(\abs{u - u'}_{\mathcal{F}_{p}}\bigr)^{(p-l) \wedge 1}\prod_{j=1}^{l}\abs{v_{j}}_{\mathcal{F}_{p}}, \notag 
			\end{align}
			where $C_{p}$ is the same constant as in \eqref{eq:pre-conti.nonlinear}.
	\end{enumerate}
\end{prop} 

Let us conclude this section by recalling from \cite[Section 6.4]{KS26} the construction of the energy measures associated with the Korevaar--Schoen energy forms. 
To this end, we need to impose another condition for $\bm{k}$. 
We say that $\bm{k} = \{ k_{r} \}_{r > 0}$ is \emph{asymptotically local} (see \cite[Definition 6.3.5]{KS26}) if and only if there exists $\{ \delta(r) \}_{r > 0} \subseteq (0,\infty)$ such that $\lim_{r \to 0}\delta(r) = 0$ and 
\begin{equation}\label{eq:asymplocal}
	\lim_{r \to 0}\int_{X}\int_{X \setminus B(y,\delta(r))}k_{r}(x,y)\,\measure(dx)\,\measure(dy) = 0. 
\end{equation}
We also need to introduce several function spaces.
We use the following notation: 
\begin{align*}
	C(X) &\coloneqq \{ u \colon X \to \mathbb{R} \mid \text{$u$ is continuous} \}, \\
	C_{b}(X) &\coloneqq \{ u \in C(X) \mid \text{$u$ is bounded} \}, \\
	C_{c}(X) &\coloneqq \{ u \in C(X) \mid \text{the support of $u$ is compact} \},
\end{align*} 
and define $\mathcal{D}_{p,b}$ as the closure of $\mathcal{F}_{p} \cap C_{b}(X)$ in $\mathcal{F}_{p}$.  
(The symbol $\mathcal{D}_{p,\infty}^{\bm{k},b}$ is used to denote $\mathcal{D}_{p,b}$ in \cite[Definition 6.4.4]{KS26}.)

Now we summarize a characterization of the associated energy measures $\Gamma_{p}\langle u \rangle$ in the following proposition. 
See \cite[Section 6.4]{KS26} for further details. 
\begin{prop}[{\cite[Theorems 6.4.3, 6.4.5 and 6.4.8]{KS26}}]\label{prop:em-constr}
	Assume that $(X,d)$ is locally compact and that the weak monotonicity property \eqref{eq:WM} holds. 
	Let $\{ r_{n} \}_{n}$ be a sequence in $(0,\infty)$ converging to $0$ as given in Lemma \ref{lem:subsec} so that the $\bm{k}$-Korevaar--Schoen $p$-energy form $(\mathcal{E}_{p},\mathcal{F}_{p})$ along $\{ r_{n} \}_{n}$ is defined by \eqref{eq:defn-multiform} with $l=0$. 
	In addition, assume that $\bm{k}$ is asymptotically local and that $\mathcal{F}_{p} \cap C_{c}(X)$ is dense in $C_{c}(X)$ with respect to the uniform norm $\norm{\,\cdot\,}_{\sup}$. 
	Then for any $u \in \mathcal{D}_{p,b}$, there exists a unique positive Radon measure $\Gamma_{p}\langle u \rangle$ on $X$, which we call the \emph{$p$-energy measure of $u$ associated with $(\mathcal{E}_{p},\mathcal{F}_{p})$}, such that for all $\varphi \in C_{c}(X)$, 
	\begin{align}\label{eq:repformula-another}
		&\int_{X}\varphi\,d\Gamma_{p}\langle u \rangle \\
		&= \lim_{n \to \infty}\int_{X}\int_{X}\abs{u(x) - u(y)}^{p}\varphi(x)k_{r_n}(x,y)\,\measure(dx)\,\measure(dy). \notag 
	\end{align}
	Furthermore, $\Gamma_{p}\langle u \rangle(X) \le \mathcal{E}_{p}(u) < \infty$. 
\end{prop} 

\section{Higher-order differentiability of energy forms}\label{sec:higher.form} 
This section is devoted to investigating higher-order (Fr\'{e}chet) differentiability of the $\bm{k}$-Korevaar--Schoen energy forms. 
For definitions and fundamental results concerning the differentiability of maps between Banach spaces, we refer the reader to \cite{AD71} or \cite[Chapter 8]{Jos}. 
In this paper, we only encounter the case where the target space is $\mathbb{R}$. 
For a $C^{m}$-map $E \colon V \to \mathbb{R}$, where $V$ is a (real) Banach space, we use $D^{m}E$ to denote the $m$-th derivative of $E$. 
Note that $D^{m}E(v)$, $v \in V$, is a (real) $m$-multilinear form defined on $V^{m}$. 
 
In this section, we always assume that $\bm{k} = \{ k_r \}_{r > 0}$ satisfies the weak monotonicity property \eqref{eq:WM} and $\{ r_{n} \}_{n}$ is a sequence in $(0,\infty)$ converging to $0$ as given in Lemma \ref{lem:subsec} so that we have the $\bm{k}$-Korevaar--Schoen $p$-energy form $(\mathcal{E}_{p},\mathcal{F}_{p})$ on $(X,\measure)$ along $\{ r_{n} \}_{n}$ defined by \eqref{eq:defn-multiform} with $l=0$. 

As in the case of $\ell^p$ and $L^p$ norms in \cite[Theorem 8]{Sun67}, the following lemma by Banach--Saks \cite{BS30} is a key ingredient.
\begin{lem}\label{lem:BS}
	There exists $C > 0$ such that for any $a,b \in \mathbb{R}$, 
	\begin{equation}\label{eq:taylor}
		\absBigg{\abs{a+b}^{p} - \sum_{l=0}^{\lfloor p \rfloor}\binom{p}{l}\abs{a}^{p-l}\mathrm{sgn}(a)^{l}b^{l}} 
		\le C\abs{b}^{p}. 
	\end{equation}
	Here $\lfloor p \rfloor \in \mathbb{N}$ is the integer part of $p$ and $\binom{p}{l}$ denotes the generalized binomial coefficient. 
\end{lem}
\begin{proof}
	This was proved in \cite[Th\'{e}or\`{e}me I]{BS30}. Let us recall the proof for reader's convenience. 
	Define 
	\[
	\varphi(x) \coloneqq \frac{1}{\abs{x}^{p}}\Biggl(\abs{1+x}^{p} - \sum_{l=0}^{\lfloor p \rfloor}\binom{p}{l}x^{l}\Biggr), \qquad x \in \mathbb{R} \setminus \{ 0 \}. 
	\]
	By Taylor's theorem, we know that 
	\[
	\abs{1+x}^{p} - \sum_{l=0}^{\lfloor p \rfloor}\binom{p}{l}x^{l} = o(\abs{x}^{p}) \qquad (x \to 0).  
	\]
	Hence the function $\varphi$ turns out to be a continuous function on $\mathbb{R}$ by putting $\varphi(0) \coloneqq 0$. 
	Furthermore, we easily see that 
	\[
	\lim_{\abs{x} \to \infty}\abs{\varphi(x)} 
	= \lim_{\abs{x} \to \infty}\absBigg{\absbigg{1+\frac{1}{x}}^{p} - \sum_{l=0}^{\lfloor p \rfloor}\binom{p}{l}\frac{x^{l}}{\abs{x}^{p}}} 
	\le 2. 
	\]
	Therefore, $C \coloneqq \sup_{x \in \mathbb{R}}\abs{\varphi(x)} < \infty$. 
	We obtain \eqref{eq:taylor} by putting $x = b/a$. 
\end{proof}

\begin{thm}\label{thm:-form.diffble} 
	The map $\mathcal{E}_{p} \colon \mathcal{F}_{p} \to [0,\infty)$ is of class $C^{m(p)}$, with its derivatives given by
	\[
	D^{l}\mathcal{E}_{p}(u) = p(p-1)\cdots(p-l+1)\mathcal{E}_{p}^{(l)}(u;\,\cdot\,), \qquad u \in \mathcal{F}_{p}, 1 \le l \le m(p). 
	\] 
	Furthermore, if $p$ is an even number, i.e., $p/2 \in \mathbb{N}$, then $\mathcal{E}_{p} \colon \mathcal{F}_{p} \to [0,\infty)$ is of class $C^{\infty}$ with $D^{p}\mathcal{E}_{p} = p!\mathcal{E}_{p}^{(p)}$ and $D^{l}\mathcal{E}_{p} = 0$ for $l > p$.
\end{thm}
\begin{proof}
	By \eqref{eq:taylor}, for any $u,v \in \mathcal{F}_{p}$ and any $r > 0$, 
	\[
	\absBigg{E_{p,r}(u+v) - E_{p,r}(u) - \sum_{l=1}^{\lfloor p \rfloor}\binom{p}{l}E_{p,r}^{(l)}(u; v,v,\dots,v)} 
	\le CE_{p,r}(v). 
	\] 
	Taking $r \to 0$ along $\{ r_n \}_{n}$ in this estimate, we obtain 
	\begin{equation}\label{eq:remainder}
		\absBigg{\mathcal{E}_{p}(u+v) - \mathcal{E}_{p}(u) - \sum_{l=1}^{\lfloor p \rfloor}\binom{p}{l}\mathcal{E}_{p}^{(l)}(u; v,v,\dots,v)} 
		\le C\mathcal{E}_{p}(v). 
	\end{equation}
	This remainder estimate together with Proposition \ref{prop:-form.basic} and a version of the converse of Taylor's theorem \cite[Theorem on p.~348]{AD71} implies that $\mathcal{E}_{p}$ is of class $C^{m(p)}$ with, for $u \in \mathcal{F}_{p}$ and $1 \le l \le m(p)$, 
	\[
	D^{l}\mathcal{E}_{p}(u) 
	= l!\binom{p}{l}\mathcal{E}_{p}^{(l)}(u;\,\cdot\,)
	= p(p-1)\cdots(p-l+1)\mathcal{E}_{p}^{(l)}(u;\,\cdot\,). 
	\]
	
	In the rest of the proof, assume that $p$ is an even number. 
	Then a direct computation shows that 
	\begin{align*}
		&E_{p,r}^{(p-1)}(u; v_{1},v_{2},\dots,v_{p-1}) - E_{p,r}^{(p-1)}(u'; v_{1},v_{2},\dots,v_{p-1}) \\
		&= E_{p,r}^{(p-1)}(u - u'; v_{1},v_{2},\dots,v_{p-1}) \\
		&= \int_{X}\int_{X}\Biggl[\bigl((u(x) - u(y)) - (u'(x) - u'(y))\bigr) \\
		&\qquad\qquad \times \Biggl(\prod_{j=1}^{p-1}(v_{j}(x) - v_{j}(y))\Biggr)k_{r}(x,y)\Biggr]\,\measure(dx)\,\measure(dy) \\
		&= E_{p,r}^{(p)}(u-u',v_{1},v_{2},\dots,v_{p-1})
	\end{align*} 
	for any $r > 0$ and any $u,u',v_{1},v_{2},\dots,v_{p-1} \in \mathcal{F}_{p}$. 
	Taking $r \to 0$ along $\{ r_n \}_{n}$ in this equality, we obtain 
	\begin{align*}
		&\mathcal{E}_{p}^{(p-1)}(u; v_{1},v_{2},\dots,v_{p-1}) - \mathcal{E}_{p}^{(p-1)}(u'; v_{1},v_{2},\dots,v_{p-1})  \\
		&= \mathcal{E}_{p}^{(p)}(u-u',v_{1},v_{2},\dots,v_{p-1}), 
	\end{align*}
	which implies that  
	\[
	D^{p}\mathcal{E}_{p}(u) 
	= p!D\mathcal{E}_{p}^{(p-1)}(u;\,\cdot\,) 
	= p!\mathcal{E}_{p}^{(p)}, \qquad u \in \mathcal{F}_{p}.  
	\]
	Since this derivative does not depend on $u$, we conclude that $D^{l}\mathcal{E}_{p} = 0$ for any $l > p$.  
\end{proof} 

\section{Higher-order differentiability of energy measures}\label{sec:higher.measures} 
As in the previous section, we next investigate the differentiability of the $\bm{k}$-Korevaar--Schoen energy measures. 
In order to follow the approach in the proof of Theorem \ref{thm:-form.diffble}, under certain assumptions, we consider the following $l$-multilinear symmetric form for each $\varphi \in C_{c}(X)$: 
\begin{align*}
	I^{(l)}(\varphi) 
	&\coloneqq \lim_{n \to \infty}\int_{X}\int_{X}\Biggl[\abs{u(x) - u(y)}^{p-l}\bigl(\mathrm{sgn}(u(x)-u(y))\bigr)^{l} \\
	& \qquad \times \Biggl(\prod_{j=1}^{l}(v_{j}(x) - v_{j}(y))\Biggr)\varphi(x)k_{r_{n}}(x,y)\Biggr]\,\measure(dx)\,\measure(dy). 
\end{align*}  
(See $\mathcal{I}_{p}^{(l)}[u; v_{1},v_{2},\dots,v_{l}](\varphi)$ below for the precise description.)
We will see that $I^{(l)}(\varphi)$ defines a bounded linear functional on $C_{c}(X)$ and hence determines the unique signed measure $\Gamma_{p}\langle u;v_{1},v_{2},\dots,v_{l} \rangle$ through the Riesz--Markov--Kakutani representation theorem. 
These associated singed measures will be used to describe higher-order derivatives of the Korevaar--Schoen energy measures. 

To carry out the strategy described above, we first make some preparations.  
Throughout this section, in addition to the setting specified at the beginning of Section \ref{sec:preli}, we assume that $(X,d)$ is locally compact. 
For $r > 0$, $l \in \mathbb{Z}$ with $1 \le l \le m(p)$, $u,v_{1},v_{2},\dots,v_{l} \in L^p(X,\measure)$ and $\varphi \in C_{b}(X)$, define  
\begin{align*}
	&I_{p,r}^{(l)}[u; v_{1},v_{2},\dots,v_{l}](\varphi) \\
	&\coloneqq \int_{X}\int_{X}\Biggl[(u(x) - u(y))^{p-l}\bigl(\mathrm{sgn}(u(x)-u(y))\bigr)^{l} \\
	&\qquad\qquad \times \Biggl(\prod_{j=1}^{l}(v_{j}(x) - v_{j}(y))\Biggr)\varphi(x)k_{r}(x,y)\Biggr]\,\measure(dx)\,\measure(dy),  
\end{align*}
and 
\[
I_{p,r}[u](\varphi)  
\coloneqq \int_{X}\int_{X}\abs{u(x) - u(y)}^{p}\varphi(x)k_{r}(x,y)\,\measure(dx)\,\measure(dy).
\] 
If $p \in \mathbb{N}$, then we also define 
\begin{align*}
	&I_{p,r}^{(p)}[v_{1},v_{2},\dots,v_{p}](\varphi) \\
	&\coloneqq \int_{X}\int_{X}\Biggl(\prod_{j=1}^{p}(v_{j}(x) - v_{j}(y))\Biggr)\varphi(x)k_{r}(x,y)\,\measure(dx)\,\measure(dy). 
\end{align*}
For ease of notation, we write $I_{p,r}^{(0)} \coloneqq I_{p,r}$, and, when $p \in \mathbb{N}$, we also write $I_{p,r}^{(l)}(u;\,\cdot\,) \coloneqq I_{p,r}^{(p)}(\,\cdot\,)$.  

The same arguments as Lemmas \ref{lem:-preform.basic} and \ref{lem:subsec} show the following lemmas. 
\begin{lem}\label{lem:-preem.basic}
	Let $r > 0$, $l \in \mathbb{Z}$ with $0 \le l \le p$ and $\varphi \in C_{b}(X)$. 
	\begin{enumerate}[\rm(a)]
		\item For any $a \in \mathbb{R}$ and any $u \in \mathcal{F}_{p}$, 
		\[
		I_{p,r}^{(l)}[au;\,\cdot\,](\varphi) = \abs{a}^{p-l}\mathrm{sgn}(a)^{l}I_{p,r}^{(l)}[u;\,\cdot\,](\varphi).  
		\]
		\item If $l \ge 1$ and $u \in \mathcal{F}_{p}$, then $I_{p,r}^{(l)}[u;\,\cdot\,](\varphi) \colon (\mathcal{F}_{p})^{l} \to \mathbb{R}$ is an $l$-multilinear symmetric form. 
		\item For any $u,v_{1},v_{2},\dots,v_{l} \in \mathcal{F}_{p}$, 
			\begin{align}\label{eq:-integral.pre-holder}
				\absBig{I_{p,r}^{(l)}[u; v_{1},v_{2},\dots,v_{l}](\varphi)} 
				&\le I_{p,r}[u](\abs{\varphi})^{\frac{p-l}{p}}\prod_{j=1}^{l}I_{p,r}[v_{j}](\abs{\varphi})^{\frac{1}{p}}  \\
				&\le \norm{\varphi}_{\sup}\bigl(\abs{u}_{\mathcal{F}_{p}}\bigr)^{p-l}\prod_{j=1}^{l}\abs{v_{j}}_{\mathcal{F}_{p}}. \notag 
			\end{align}
		\item For any $u,u',v_{1},v_{2},\dots,v_{l} \in \mathcal{F}_{p}$, 
			\begin{align}\label{eq:-integral.pre-conti.nonlinear}
				&\absBig{I_{p,r}^{(l)}[u;v_{1},v_{2},\dots,v_{l}](\varphi) - I_{p,r}^{(l)}[u';v_{1},v_{2},\dots,v_{l}](\varphi)} \\
				&\quad \le C_{p}\Bigl[I_{p,r}[u](\abs{\varphi})^{\frac{1}{p}} \vee I_{p,r}[u'](\abs{\varphi})^{\frac{1}{p}}\Bigr]^{(p-l-1)^{+}} \notag \\
				&\quad\quad  \times \bigl(I_{p,r}[u-u'](\abs{\varphi})^{\frac{1}{p}}\bigr)^{(p-l) \wedge 1} \prod_{j=1}^{l}I_{p,r}[v_j](\abs{\varphi})^{\frac{1}{p}} \notag \\
				&\quad \le C_{p}\norm{\varphi}_{\sup}\Bigl[\abs{u}_{\mathcal{F}_{p}} \vee \abs{u'}_{\mathcal{F}_{p}}\Bigr]^{(p-l-1)^{+}}\bigl(\abs{u - u'}_{\mathcal{F}_{p}}\bigr)^{(p-l) \wedge 1}\prod_{j=1}^{l}\abs{v_{j}}_{\mathcal{F}_{p}}. \notag 
			\end{align}
			where $C_p$ is the same constant as in \eqref{eq:pre-conti.nonlinear}.
	\end{enumerate}
\end{lem}

\begin{lem}\label{lem:-integral.subsec}
	Assume that the weak monotonicity property \eqref{eq:WM} holds. 
	Let $\varphi \in C_{b}(X)$ and  $\{ \widetilde{r}_{n} \}_{n \ge 1}$ be a sequence in $(0,\infty)$ such that $\lim_{n \to \infty}\widetilde{r}_{n} = 0$. 
	Then there exists a subsequence $\{ r_{n} \}_{n}$ of $\{ \widetilde{r}_{n} \}_{n}$ such that, for any $0 \le l \le m(p)$ and any $u,v_{1},v_{2},\dots,v_{l} \in \mathcal{F}_{p}$, the sequence $\bigl\{ I_{p,r_{n}}^{(l)}[u; v_{1},v_{2},\dots,v_{l}](\varphi) \bigr\}_{n}$ is convergent in $\mathbb{R}$. 
\end{lem} 

\begin{defn}\label{defn:-integral.KS}
	Assume that the weak monotonicity property \eqref{eq:WM} holds and that $\{ \widetilde{r}_{n} \}_{n \ge 1}$ is a sequence in $(0,\infty)$ converging to $0$. 
	Let $\varphi \in C_{b}(X)$ and $\{ r_{n} \}_{n}$ be a subsequence of $\{ \widetilde{r}_{n} \}_{n \ge 1}$ as given in Lemma \ref{lem:subsec}. (Note that the sequence $\{ r_{n} \}_{n}$ may depend on $\varphi$.)
	For $0 \le l \le m(p)$, define $\mathcal{I}_{p}^{(l)}[\,\cdot\,](\varphi) \colon (\mathcal{F}_{p})^{l+1} \to \mathbb{R}$ by 
	\begin{equation}\label{eq:defn.integralform}
		\mathcal{I}_{p}^{(l)}[u; v_{1},v_{2},\dots,v_{l}](\varphi) \coloneqq \lim_{n\to\infty}I_{p,r_{n}}^{(l)}[u; v_{1},v_{2},\dots,v_{l}](\varphi) 
	\end{equation} 
	for $u,v_{1},v_{2},\dots,v_{l} \in \mathcal{F}_{p}$. 
	We also write $\mathcal{I}_{p}[u](\varphi) \coloneqq \mathcal{I}_{p}^{(0)}[u](\varphi)$. 
\end{defn}

Similarly to Proposition \ref{prop:-form.basic}, the following proposition is immediate from the definition in \eqref{eq:defn.integralform} and Lemma \ref{lem:-preem.basic}.  
\begin{prop}\label{prop:-integral.basic}
	Assume that the weak monotonicity property \eqref{eq:WM} holds and that $\{ \widetilde{r}_{n} \}_{n \ge 1}$ is a sequence in $(0,\infty)$ converging to $0$. 
	Let $\varphi \in C_{b}(X)$ and $\{ r_{n} \}_{n}$ be a subsequence of $\{ \widetilde{r}_{n} \}_{n \ge 1}$ as given in Lemma \ref{lem:subsec}, let $l \in \mathbb{Z}$ with $0 \le l \le p$ and let $\mathcal{I}_{p}^{(l)}[\,\cdot\,](\varphi) $ be as defined in \eqref{eq:defn.integralform}. 
	\begin{enumerate}[\rm(a)]
		\item For any $a \in \mathbb{R}$ and any $u \in \mathcal{F}_{p}$, 
		\[
		\mathcal{I}_{p}^{(l)}[au;\,\cdot\,](\varphi) = \abs{a}^{p-l}\mathrm{sgn}(a)^{l}\mathcal{I}_{p}^{(l)}[u;\,\cdot\,](\varphi). 
		\]
		\item If $l \ge 1$, then, for any $u \in \mathcal{F}_{p}$, $\mathcal{I}_{p}^{(l)}[u;\,\cdot\,](\varphi) \colon (\mathcal{F}_{p})^{l} \to \mathbb{R}$ is an $l$-multilinear symmetric form. 
		\item For any $u,v_{1},v_{2},\dots,v_{l} \in \mathcal{F}_{p}$, 
			\begin{align}\label{eq:-integral.holder}
				\absBig{\mathcal{I}_{p}^{(l)}[u; v_{1},v_{2},\dots,v_{l}](\varphi)} 
				&\le \mathcal{I}_{p}[u](\abs{\varphi})^{\frac{p-l}{p}}\prod_{j=1}^{l}\mathcal{I}_{p}[v_{j}](\abs{\varphi})^{\frac{1}{p}} \\
				&\le \norm{\varphi}_{\sup}\bigl(\abs{u}_{\mathcal{F}_{p}}\bigr)^{p-l}\prod_{j=1}^{l}\abs{v_{j}}_{\mathcal{F}_{p}}. \notag 
			\end{align} 
		\item For any $u,u',v_{1},v_{2},\dots,v_{l} \in \mathcal{F}_{p}$, 
			\begin{align}\label{eq:-integral.conti.nonlinear}
				&\absBig{\mathcal{I}_{p}^{(l)}[u;v_{1},v_{2},\dots,v_{l}](\varphi) - \mathcal{I}_{p}^{(l)}[u';v_{1},v_{2},\dots,v_{l}](\varphi)} \\
				&\quad \le C_{p}\Bigl[\mathcal{I}_{p}[u](\abs{\varphi})^{\frac{1}{p}} \vee \mathcal{I}_{p}[u'](\abs{\varphi})^{\frac{1}{p}}\Bigr]^{(p-l-1)^{+}} \notag \\
				&\quad \quad \times \bigl(\mathcal{I}_{p}[u-u'](\abs{\varphi})^{\frac{1}{p}}\bigr)^{(p-l) \wedge 1}\prod_{j=1}^{l}\mathcal{I}_{p}[v_j](\abs{\varphi})^{\frac{1}{p}} \notag \\
				&\quad \le C_{p}\norm{\varphi}_{\sup}\Bigl[\abs{u}_{\mathcal{F}_{p}} \vee \abs{u'}_{\mathcal{F}_{p}}\Bigr]^{(p-l-1)^{+}}\bigl(\abs{u - u'}_{\mathcal{F}_{p}}\bigr)^{(p-l) \wedge 1}\prod_{j=1}^{l}\abs{v_{j}}_{\mathcal{F}_{p}}, \notag 
			\end{align}
			where $C_{p}$ is the same constant as in \eqref{eq:-integral.pre-conti.nonlinear}. 
	\end{enumerate}
\end{prop} 

A trouble here is that the subsequence $\{ r_{n} \}_{n}$ used in \eqref{eq:defn.integralform} may depend on $\varphi$. 
Because of this, the linearity of $\varphi \mapsto \mathcal{I}_{p}^{(l)}[u; v_{1},v_{2},\dots,v_{l}](\varphi)$ can not be ensured at this stage. 
To overcome this issue, we need to choose the sequence $\{ r_n \}_{n}$ used in the constructions of $\mathcal{I}_{p}^{(l)}$ and $\mathcal{I}_{p}$ to be the same as the sequence used to define the Korevaar--Schoen energy measures in \eqref{eq:repformula-another}.

The precise setting and assumptions are as follows. 
In the rest of the paper, we assume that the assumptions of Proposition \ref{prop:em-constr} hold; namely $\bm{k} = \{ k_r \}_{r > 0}$ satisfies the weak monotonicity property \eqref{eq:WM} and is asymptotically local (recall \eqref{eq:asymplocal}), $\mathcal{F}_{p} \cap C_{c}(X)$ is dense in $C_{c}(X)$, and $\{ \delta_n \}_{n}$ is a sequence in $(0,\infty)$ converging to $0$ such that the equality \eqref{eq:repformula-another} with $\delta_n$ in place of $r_n$ holds for all $\varphi \in C_{c}(X)$. 
(As in Proposition \ref{prop:em-constr}, we use $\Gamma_{p}\langle u \rangle$ to denote the $p$-energy measure of $u \in \mathcal{D}_{p,b}$ associated with the $\bm{k}$-Korevaar--Schoen energy form along $\{ \delta_n \}_{n}$.) 
Now we take $\{ \delta_n \}_{n}$ as the sequence $\{ \widetilde{r}_{n} \}_{n}$ in Definition \ref{defn:-integral.KS} and define $\mathcal{I}_{p}^{(l)}[u; v_{1},v_{2},\dots,v_{l}](\varphi)$ and $\mathcal{I}_{p}[u](\varphi)$ for $0 \le l \le p$ and $u,v_1,v_2,\dots,v_l \in \mathcal{F}_{p}$ accordingly.  
Note that if $u \in \mathcal{D}_{p,b}$ and $\varphi \in C_{c}(X)$, then we have from \eqref{eq:repformula-another} and \eqref{eq:defn.integralform} that 
\begin{equation}\label{eq:identity}
	\mathcal{I}_{p}[u](\varphi) = \int_{X}\varphi\,d\Gamma_{p}\langle u \rangle.  
\end{equation} 
 
The following lemma ensures the desired linearity. 
\begin{lem}\label{lem:-integral.linear}
	Under the above situation, for any $l \in \mathbb{Z}$ with $1 \le l \le m(p)$, any $\varphi \in C_{c}(X)$ and any $u,v \in \mathcal{D}_{p,b}$, 
	\begin{equation}\label{eq:-integral.l-derivative}
		\mathcal{I}_{p}^{(l)}[u; v,v,\dots,v](\varphi)
		= \frac{1}{p(p-1) \cdots (p-l+1)}\frac{d^{l}}{dt^{l}}\int_{X}\varphi\,d\Gamma_{p}\langle u + tv \rangle\bigg|_{t=0}  
	\end{equation}
	Consequently, for any $u,v_{1},v_{2},\dots,v_{l} \in \mathcal{D}_{p,b}$, 
	\begin{equation}\label{eq:-integral.limit}
		\mathcal{I}_{p}^{(l)}[u; v_{1},v_{2},\dots,v_{l}](\varphi)
		= \lim_{n \to \infty}I_{p,\delta_{n}}^{(l)}[u; v_{1},v_{2},\dots,v_{l}](\varphi)
	\end{equation}
	If $p$ is even, then \eqref{eq:-integral.l-derivative} and \eqref{eq:-integral.limit} hold in the case of $l = p$ as well. 
\end{lem}
\begin{proof}
	Fix $\varphi \in C_{c}(X)$. 
	Let $\{ \delta_{n_k} \}_{k}$ be any subsequence of $\{ \delta_{n} \}_{n}$ such that the following limit exists in $\mathbb{R}$ for any $0 \le l \le p$ and  any $u,v \in \mathcal{F}_{p}$ (recall Lemma \ref{lem:-integral.subsec}): 
	\begin{equation}\label{eq:choice.limit}
		\lim_{k \to \infty}I_{p,\delta_{n_k}}^{(l)}[u; v,v,\dots,v](\varphi) 
		\eqqcolon \widetilde{\mathcal{I}}_{p}^{(l)}[u; v,v,\dots,v](\varphi). 
	\end{equation}
	Now assume $u,v \in \mathcal{D}_{p,b}$. 
	Similarly to \eqref{eq:remainder}, by virtue of \eqref{eq:taylor}, we can obtain 
	\begin{align}\label{eq:-integral.remainder}
		&\absBigg{\mathcal{I}_{p}[u+v](\varphi) - \mathcal{I}_{p}[u](\varphi) - \sum_{l=1}^{\lfloor p \rfloor}\binom{p}{l}\widetilde{\mathcal{I}}_{p}^{(l)}[u; v,v,\dots,v](\varphi)} \\
		&\quad \le C\,\mathcal{I}_{p}[v]\bigl(\abs{\varphi}\bigr) 
		\le C\norm{\varphi}_{\sup}\mathcal{E}_{p}(v). \notag 
	\end{align}
	This together with \eqref{eq:identity} and the converse of Taylor's theorem (e.g., \cite[Theorem on p.~348]{AD71}) implies that for $1 \le l \le m(p)$, 
	\[
	\widetilde{\mathcal{I}}_{p}^{(l)}[u; v,v,\dots,v](\varphi)
		= \frac{1}{p(p-1) \cdots (p-l+1)}\frac{d^{l}}{dt^{l}}\mathcal{I}_{p}[u+tv](\varphi)\bigg|_{t=0}. 
	\]
	Hence \eqref{eq:-integral.l-derivative} holds.  
	Furthermore, as the choice of subsequence $\{ \delta_{n_k} \}_{k}$ is arbitrary as long as the limit in \eqref{eq:choice.limit} exists, we obtain \eqref{eq:-integral.limit} in the case of $v_{j} = v \, (1 \le j \le l)$. 
	Since any (real) multilinear form can be expressed in terms of its diagonal part by the polarization formula (see, e.g., \cite[Theorem 1 or Eq. (7)]{Tho14}), the equality \eqref{eq:-integral.limit} follows. 
	
	When $p$ is even, by following an argument similar to that in the proof of Theorem \ref{thm:-form.diffble}, one can obtain the following equality: 
	\begin{align}\label{eq:identity.even}
		&\widetilde{\mathcal{I}}_{p}^{(p-1)}[u; v_{1},v_{2},\dots,v_{p-1}](\varphi) - \widetilde{\mathcal{I}}_{p}^{(p-1)}[u'; v_{1},v_{2},\dots,v_{p-1}](\varphi)  \\
		&= \widetilde{\mathcal{I}}_{p}^{(p)}[u-u',v_{1},v_{2},\dots,v_{p-1}](\varphi), \notag 
	\end{align}
	This is enough for proving \eqref{eq:-integral.l-derivative} and \eqref{eq:-integral.limit} in the case of $l = p$. 
\end{proof}

We can now associate a unique finite signed Borel measure with the bounded linear functional $\varphi \mapsto \mathcal{I}_{p}^{(l)}[u; v_{1},v_{2},\dots,v_{l}](\varphi)$.
\begin{prop}\label{prop:multilinear.em}
	Assume the same assumptions as in Lemma \ref{lem:-integral.linear}. 
	Let $l \in \mathbb{Z}$ with $1 \le l \le p$. 
	For any $u,v_1,v_2,\dots,v_l \in \mathcal{D}_{p,b}$, there exists a unique finite signed Borel measure\footnote{We can show that a signed Radon measure $\mu$ on $X$ satisfying \eqref{eq:defn-multilinear.em} has to satisfy $\abs{\mu}(X) < \infty$. As a result, the same conclusion on the uniqueness holds even if we replace the condition ``finite signed Borel measure'' by ``signed Radon measure''. We omit the proof to avoid unnecessary technical details.} $\Gamma_{p}^{(l)}\langle u;v_1,v_2,\dots,v_l \rangle$ on $X$ such that  
	\begin{equation}\label{eq:defn-multilinear.em}
		\int_{X}\varphi\,d\Gamma_{p}^{(l)}\langle u;v_1,v_2,\dots,v_l \rangle 
		= \mathcal{I}_{p}^{(l)}[u; v_{1},v_{2},\dots,v_{l}](\varphi)
	\end{equation}
	for all $\varphi \in C_{c}(X)$.
	Moreover, the following hold: 
	\begin{enumerate}[\rm(a)]
		\item\label{it:multiem-homog} For any $a \in \mathbb{R}$ and any $u \in \mathcal{D}_{p,b}$, the equality $\Gamma_{p}^{(l)}\langle au;\,\cdot\, \rangle = \abs{a}^{p-l}\mathrm{sgn}(a)^{l}\Gamma_{p}^{(l)}\langle u;\,\cdot\, \rangle$ as signed measures hold.
		\item\label{it:multiem-form} Let $A$ be a Borel set of $X$. 
			If $l \ge 1$, then, for any $u \in \mathcal{D}_{p,b}$, $\Gamma_{p}^{(l)}\langle u;\,\cdot\, \rangle(A) \colon (\mathcal{D}_{p,b})^{l} \to \mathbb{R}$ is an $l$-multilinear symmetric form. 
	\end{enumerate}
\end{prop}
\begin{proof}
	By virtue of \eqref{eq:-integral.limit} and \eqref{eq:-integral.holder}, the map $\Psi \colon C_{c}(X) \to \mathbb{R}$ defined as 
	\[
	\Psi(\varphi) \coloneqq \mathcal{I}_{p}^{(l)}[u; v_{1},v_{2},\dots,v_{l}](\varphi)
	\]
	is a bounded linear functional. 
	By a standard argument using \eqref{eq:-integral.holder}, we can uniquely extend $\Psi$ to a bounded linear functional $\widetilde{\Psi}$ on $C_{0}(X)$, where $C_{0}(X)$ is the set of continuous functions that vanish at infinity \cite[Definition 3.16]{Rud}. 
	By applying the Riesz--Markov--Kakutani representation theorem \cite[Theorem 6.19]{Rud} to $\widetilde{\Psi}$, we obtain a unique finite signed Borel measure $\Gamma_{p}^{(l)}\langle u;v_1,v_2,\dots,v_l \rangle$ on $X$ such that 
	\begin{equation}\label{eq:extRMK}
		\widetilde{\Psi}(\varphi) = \int_{X}\varphi\,d\Gamma_{p}^{(l)}\langle u;v_1,v_2,\dots,v_l \rangle \qquad \text{for all $\varphi \in C_{0}(X)$,}
	\end{equation}
	which implies \eqref{eq:defn-multilinear.em}. 
	
	To prove the claimed uniqueness of $\Gamma_{p}^{(l)}\langle u;v_1,v_2,\dots,v_l \rangle$, let $\mu$ be another finite singed Borel measure on $X$ satisfying \eqref{eq:defn-multilinear.em}. 
	Since $\mu$ is a finite signed measure on $X$, the functional $I_{\mu}$ given by $I_{\mu}(\varphi) \coloneqq \int_{X}\varphi\,d\mu$, $\varphi \in C_{0}(X)$, is a bounded linear functional on $C_{0}(X)$ that coincides with $\Psi$ on  $C_{c}(X)$. 
	Thus we have $I_{\mu} = \widetilde{\Psi}$.
	The desired uniqueness $\mu = \Gamma_{p}^{(l)}\langle u;v_1,v_2,\dots,v_l \rangle$ follows from the Riesz--Markov--Kakutani representation theorem. 
	The uniqueness through the representation \eqref{eq:defn-multiform} immediately implies \ref{it:multiem-homog} and \ref{it:multiem-form}. 
	The proof is completed. 
\end{proof}

\begin{prop}\label{prop:genintform}
	Assume the same assumptions as in Lemma \ref{lem:-integral.linear}. 
	Let $l \in \mathbb{Z}$ with $1 \le l \le p$ and $\varphi, \psi_{1},\psi_{2},\dots,\psi_{l} \colon X \to [0,\infty)$ be bounded Borel measurable functions.   
	\begin{enumerate}[\rm(a)]
		\item\label{it:genintform-form} For any $u \in \mathcal{D}_{p,b}$, $\int_{X}\varphi\,d\Gamma_{p}^{(l)}\langle u; \,\cdot\, \rangle \colon (\mathcal{D}_{p,b})^{l} \to [0,\infty)$ is an $l$-multilinear symmetric form. 
		\item\label{it:genintform-holder} For any $u,v_{1},v_{2},\dots,v_{l} \in \mathcal{D}_{p,b}$, 
			\begin{align}\label{eq:em-holder}
				&\int_{X}\varphi\left(\prod_{j=1}^{l}\psi_{j}\right)\,d\absBig{\Gamma_{p}^{(l)}\langle u;v_{1},v_{2},\dots,v_{l}\rangle} \\
				&\quad \le \left(\int_{X}\varphi^{\frac{p}{p-l}}\,d\Gamma_{p} \langle u \rangle\right)^{\frac{p-l}{p}}\prod_{j=1}^{l}\left(\int_{X}\psi_{j}\,d\Gamma_{p}\langle v_{j} \rangle\right)^{\frac{1}{p}}. \notag 
			\end{align} 
		\item\label{it:genintform-nconti} For any $u_{1},u_{2},v_{1},v_{2},\dots,v_{l} \in \mathcal{D}_{p,b}$, 
			\begin{align}\label{eq:em-ncont}
				&\absbigg{\int_{X}\varphi\,d\Gamma_{p}^{(l)}\langle u_{1};v_{1},v_{2},\dots,v_{l} \rangle - \int_{X}\varphi\,d\Gamma_{p}^{(l)}\langle u_{2};v_{1},v_{2},\dots,v_{l} \rangle}  \\
				&\le C_{p}\Biggl[\max_{i \in \{ 1,2 \}}\int_{X}\varphi\,d\Gamma_{p}\langle u_{i} \rangle\Biggr]^{\frac{(p-l-1)^{+}}{p}}\left(\int_{X}\varphi\,d\Gamma_{p}\langle u_{1}-u_{2} \rangle\right)^{\frac{(p-l) \wedge 1}{p}} \notag \\
				&\quad \times \prod_{j=1}^{l}\left(\int_{X}\varphi\,d\Gamma_{p}\langle v_j \rangle\right)^{\frac{1}{p}},  \notag 
			\end{align}
			where $C_{p}$ is the same constant as in \eqref{eq:pre-conti.nonlinear}.  
	\end{enumerate}
\end{prop}
\begin{proof}
	\ref{it:genintform-form}: This is immediate from Proposition \ref{prop:multilinear.em}~\ref{it:multiem-form} and a standard approximation of $\varphi$ by a non-decreasing sequence of simple functions. 
	
	\ref{it:genintform-holder}: Let $\nu$ be the finite Borel measure on $X$ given by  
	\[
	\nu \coloneqq \absBig{\Gamma_{p}^{(l)}\langle u;v_{1},v_{2},\dots,v_{l}\rangle} + \Gamma_{p}\langle u \rangle + \sum_{j=1}^{l}\Gamma_{p}\langle v_{j} \rangle.
	\]
	Since $C_{c}(X)$ is dense in $L^1(X,\nu)$, we have from \eqref{eq:-integral.holder} that for all Borel sets $A$ of $X$, 
	\begin{equation}\label{eq:-em.setwiseholder}
		\absBig{\Gamma_{p}^{(l)}\langle u; v_{1},v_{2},\dots,v_{l} \rangle(A)} 
		\le \Gamma_{p}\langle u \rangle(A)^{\frac{p-l}{p}}\prod_{j=1}^{l}\Gamma_{p}\langle v_{j} \rangle(A)^{\frac{1}{p}}. 
	\end{equation}
	Let $X = \mathcal{P} \cup \mathcal{N}$ be the Hahn decomposition with respect to $\Gamma_{p}^{(l)}\langle u;v_{1},v_{2},\dots,v_{l}\rangle$, i.e., $\mathcal{P}$ and $\mathcal{N}$ are Borel sets of $X$ such that $\mathcal{P} \cap \mathcal{N} = \emptyset$, $\Gamma_{p}^{(l)}\langle u;v_{1},v_{2},\dots,v_{l}\rangle(A \cap \mathcal{P}) \ge 0$ and $\Gamma_{p}^{(l)}\langle u;v_{1},v_{2},\dots,v_{l}\rangle(A \cap \mathcal{N}) \le 0$ for all Borel sets $A$ of $X$. 
	Note that the total variation measure $\absbig{\Gamma_{p}^{(l)}\langle u;v_{1},v_{2},\dots,v_{l}\rangle}$ is given by the following Jordan decomposition: For any Borel set $A$ of $X$, 
	\begin{align*}
		&\absBig{\Gamma_{p}^{(l)}\langle u;v_{1},v_{2},\dots,v_{l}\rangle}(A)  \\
		&= \Gamma_{p}^{(l)}\langle u;v_{1},v_{2},\dots,v_{l}\rangle(A \cap \mathcal{P}) - \Gamma_{p}^{(l)}\langle u;v_{1},v_{2},\dots,v_{l}\rangle(A \cap \mathcal{N}). 
	\end{align*}
	Using the Jordan decomposition above, \eqref{eq:-em.setwiseholder} and H\"{o}lder's inequality, we can show that 
	\begin{equation}\label{eq:-em.setwiseholder-tv}
		\absBig{\Gamma_{p}^{(l)}\langle u; v_{1},v_{2},\dots,v_{l} \rangle}(A) 
		\le \Gamma_{p}\langle u \rangle(A)^{\frac{p-l}{p}}\prod_{j=1}^{l}\Gamma_{p}\langle v_{j} \rangle(A)^{\frac{1}{p}}. 
	\end{equation}
	For details of an argument similar to this one, see \cite[(4.14) in the proof of Proposition 4.8]{KS.gc}. 
	Now we can obtain \eqref{eq:em-holder} by combining \eqref{eq:-em.setwiseholder-tv} and approximations of $\varphi, \psi_{1},\psi_{2},\dots,\psi_{l}$ by simple functions; see \cite[Proof of Proposition 4.8]{KS.gc} for details. 
	
	\ref{it:genintform-nconti}: We can easily extend \eqref{eq:-integral.conti.nonlinear} to \eqref{eq:em-ncont} since $C_{c}(X)$ is dense in $L^{1}(X,\lambda)$ for the finite Borel measure $\lambda$ on $X$ given by 
	\[
	\lambda \coloneqq \sum_{i \in \{ 1,2 \}}\left(\absBig{\Gamma_{p}^{(l)}\langle u_{i}; v_{1},v_{2},\dots,v_{l} \rangle} + \Gamma_{p}\langle u_i \rangle\right) + \sum_{j=1}^{l}\Gamma_{p}\langle v_j \rangle, 
	\]
	so the proof is completed. 
\end{proof}

Finally, we can show the following main result. 
\begin{thm}\label{thm:em-derivatives}
	Assume the same assumptions as in Lemma \ref{lem:-integral.linear}. 
	Let $\varphi \colon X \to (-\infty,\infty)$ be a bounded Borel measurable function. 
	Then the map $\int_{X}\varphi\,d\Gamma_{p}\langle \,\cdot\, \rangle \colon \mathcal{D}_{p,b} \to [0,\infty)$ is of class $C^{m(p)}$, with its derivatives given by
	\[
	\left(D^{l}\int_{X}\varphi\,d\Gamma_{p}\langle \,\cdot\, \rangle\right)(u) 
	= p(p-1)\cdots(p-l+1) \int_{X}\varphi\,d\Gamma_{p}^{(l)}\langle u;\,\cdot\,\rangle
	\] 
	for $u \in \mathcal{D}_{p,b}$ and $1 \le l \le m(p)$. 
	Furthermore, if $p$ is even, then $\int_{X}\varphi\,d\Gamma_{p}\langle \,\cdot\, \rangle \colon \mathcal{D}_{p,b} \to [0,\infty)$ is of class $C^{\infty}$ with 
	\[
	D^{p}\int_{X}\varphi\,d\Gamma_{p}\langle \,\cdot\, \rangle = p!\int_{X}\varphi\,d\Gamma_{p}^{(p)}\langle \,\cdot\, \rangle
	\]
	and $D^{l}\int_{X}\varphi\,d\Gamma_{p}\langle \,\cdot\, \rangle = 0$ for $l > p$.
\end{thm} 
\begin{proof}
	Let $u,v \in \mathcal{D}_{p,b}$. 
	Since $C_{c}(X)$ is dense in $L^1(X,\nu)$ for the finite Borel measure $\nu$ on $X$ given by 
	\[
	\nu \coloneqq \Gamma_{p}\langle u+v \rangle + \Gamma_{p}\langle u \rangle + \sum_{l=1}^{\lfloor p \rfloor}\absBig{\Gamma_{p}^{(l)}\langle u; v,v,\dots,v \rangle}, 
	\]
	we can easily extend \eqref{eq:-integral.remainder} to the following estimate by using \eqref{eq:-integral.limit}: 
	\begin{align*}
		&\absBigg{\int_{X}\varphi\,d\Gamma_{p}\langle u+v \rangle - \int_{X}\varphi\,d\Gamma_{p}\langle u \rangle - \sum_{l=1}^{\lfloor p \rfloor}\binom{p}{l}\int_{X}\varphi\,d\Gamma_{p}^{(l)}\langle u;v,v,\dots,v \rangle} \\
		&\le C\norm{\varphi}_{\sup}\mathcal{E}_{p}(v). 
	\end{align*}
	Once this estimate is established, the rest of the argument is the same as in the proof of Theorem \ref{thm:-form.diffble}. The case where $p$ is even can also be shown by a direct computation, as in the proof of Theorem \ref{thm:-form.diffble}. 
\end{proof}


\end{document}